\documentclass[12pt,leqno]{amsart}
\usepackage{amssymb,url}
\usepackage{rotating}
\usepackage{pdflscape}
\newcommand{\Z}{{\mathbb{Z}}}

\newcommand{\fI}{{\mathfrak{I}}}
\newcommand{\fS}{{\mathfrak{S}}}

\newcommand{\cD}{{\mathcal{D}}}

\newcommand{\bH}{{\mathcal{H}}}
\newcommand{\cI}{{\mathcal{I}}}

\newcommand{\cR}{{\mathcal{R}}}

\newcommand{\cT}{{\mathfrak{T}}}

\newcommand{\Ind}{{\operatorname{Ind}}}
\newcommand{\prI}{{\operatorname{pr}}}
\renewcommand{\leq}{\leqslant}
\renewcommand{\geq}{\geqslant}

\newtheorem{thm}{Theorem}[section]

\newtheorem{conj}[thm]{Conjecture}
\newtheorem{cor}[thm]{Corollary}
\newtheorem{prop}[thm]{Proposition}

\theoremstyle{definition}
\newtheorem{exmp}[thm]{Example}
\newtheorem{defn}[thm]{Definition}

\theoremstyle{remark}
\newtheorem{rem}[thm]{Remark}

\begin{document}

\title{A generalised $\tau$-invariant for the unequal parameter case}

\date{}

\author{Meinolf Geck}
\address{Fachbereich Mathematik, IAZ - Lehrstuhl f\"ur Algebra, 
Universit\"at Stuttgart, Pfaf\-fen\-wald\-ring 57, 70569 Stuttgart, Germany}
\email{meinolf.geck@mathematik.uni-stuttgart.de}


\begin{abstract}
In 1979, Vogan proposed a generalised $\tau$-invariant for characterising
primitive ideals in enveloping algebras. Via a known dictionary this
translates to an invariant of left cells of finite Weyl groups. Although 
it is not a complete invariant, it is extremely useful in describing left 
cells. Here, we propose a general framework for defining such invariants
which also applies to Hecke algebras with unequal parameters.
\end{abstract}

\maketitle

\pagestyle{myheadings}
\markboth{Geck}{A generalised $\tau$-invariant for the unequal parameter case}

\section{Introduction} \label{sec:intro}

Let $W$ be a finite Weyl group. Using the corresponding generic 
Iwahori--Hecke algebra and the ''new'' basis of this algebra introduced 
by Kazhdan and Lusztig \cite{KL}, we obtain partitions of $W$ into left,
right and two-sided cells. Analogous notions originally arose in the 
theory of primitive ideals in enveloping algebras; see Joseph \cite{jos}. 
This is one of the sources for the interest in knowing the cell partitions 
of $W$. Vogan \cite{voga}, \cite{voga1} introduced invariants of left cells
which are computable in terms of certain combinatorially defined operators 
$T_{\alpha\beta}$, $S_{\alpha\beta}$ where $\alpha,\beta$ are adjacent 
simple roots of $W$. In the case where $W$ is the symmetric group $\fS_n$, 
these invariants completely characterise the left cells; see \cite[\S 5]{KL},
\cite[\S 6]{voga}. Although Vogan's invariants are not complete invariants 
in general, they have turned out to be extremely useful in describing left 
cells; see, most notably, the work of Garfinkle \cite{gar1}, \cite{gar2},
\cite{gar3}.

Now, the Kazhdan--Lusztig cell partitions are not only defined and
interesting for finite Weyl groups, but also for affine Weyl groups
and Coxeter groups in general; see, e.g., Lusztig \cite{Lu1}, 
\cite{Lusztig03}. Furthermore, the original theory was extended by 
Lusztig \cite{Lusztig83} to allow the possibility of attaching weights to 
the simple reflections. The original setting then corresponds to the case 
where all weights are equal to $1$; we will refer to this case as the
''equal parameter case''. Using ideas from Lusztig \cite[\S 10]{Lu1}, our 
aim here is to propose analogues of Vogan's invariants which work in 
general, i.e., for arbitrary Coxeter groups and arbitrary (positive) 
weights. 

In Section~\ref{sec1} we briefly recall the basic set-up concerning
Iwahori--Hecke algebras and cells in the sense of Kazhdan and Lusztig. 
As Vogan's orginal definition of the generalised $\tau$-invariant 
relies on the theory of primitive ideals, it only applies to finite Weyl 
groups. In Section~\ref{seceq}, we show how to translate this into the 
setting of Kazhdan and Lusztig. (A similar translation has also been done 
by Shi \cite[4.2]{shi}, who uses a definition slightly different from
Vogan \cite{voga}; our argument seems to be more direct.) Thus, the 
generalised $\tau$-invariant is available for arbitrary Coxeter groups in 
the equal parameter case. In Section~\ref{sec2}, we propose an abstract 
setting for defining such invariants; this essentially relies on the 
concept of ''induction of cells'' \cite{myind}, \cite{myrel} and Lusztig's 
method of ''strings'' \cite[\S 10]{Lu1}. In Theorem~\ref{mythm} we show that 
this gives indeed rise to new invariants of left cells. As a 
by-product of our approach, we obtain new (and less computational) proofs 
of the results concerning the ''star'' operations in \cite[\S 4]{KL} and 
the analogous results for ''strings'' in \cite[\S 10]{Lu1}. We conclude 
by discussing examples and stating open problems.

\section{Weight functions and cells} \label{sec1}

Let $W$ be a Coxeter group with generating set $S$ and corresponding
length function $\ell\colon W\rightarrow \Z_{\geq 0}$. Let $\pi=\{p_s \mid 
s \in S\} \subseteq \Z$ be a set of ``weights'' where $p_s=p_t$ 
whenever $s,t\in S$ are conjugate in $W$. This gives rise to a weight 
function $p\colon W \rightarrow \Z$ in the sense of Lusztig 
\cite{Lusztig03}; for $w \in W$, we have $p_w=p_{s_1}+\ldots +p_{s_k}$ 
where $w= s_1\cdots s_k$ ($s_i \in S$) is a reduced expresssion
for~$w$. The original setup in \cite{KL} corresponds to the case where 
$p_s=1$ for all $s \in S$; this will be called the ''equal parameter 
case''.  We shall assume throughout that $p_s>0$ for all $s\in S$.
(There are standard techniques for reducing the general case to this case; 
see Bonnaf\'e \cite[\S 2]{bo3}.)

Let $\bH=\bH_A(W,S,\{p_s\})$ be the corresponding generic Iwahori--Hecke 
algebra, where $A=\Z[v,v^{-1}]$ is the ring of Laurent polynomials in an 
indeterminate $v$. This algebra is free over $A$ with basis $\{T_w\mid 
w \in W\}$, and the multiplication is given by the rule
\[ T_sT_w=\left\{\begin{array}{cl} T_{sw} & \quad \mbox{if $sw>w$},\\
T_{sw}+(v^{p_s}-v^{-p_s})T_w & \quad \mbox{if $sw<w$},
\end{array}\right.\]
where $s\in S$ and $w\in W$; here, $\leq$ denotes the Bruhat--Chevalley
order on $W$. 

Let $\{C_w'\mid w\in W\}$ be the ''new'' basis of $\bH$ introduced in 
\cite[(1.1.c)]{KL}, \cite[\S 2]{Lusztig83}. (These basis elements are
denoted $c_w$ in \cite{Lusztig03}.) For any $x,y\in W$, we write 
\[ C_x'\,C_y'=\sum_{z\in W} h_{x,y,z} \, C_z' \qquad \mbox{where
$h_{x,y,z} \in A$ for all $x,y,z\in W$}.\]
We have the following more explicit formula for $s\in S$, $y\in W$
(see \cite[\S 6]{Lusztig83}, \cite[Chap.~6]{Lusztig03}):
\[C_s'\,C_y' = \left\{\begin{array}{ll} \displaystyle{(v^{p_s}+
v^{-p_s})\,C_y'} &\quad \mbox{if $sy<y$},\\ 
\displaystyle{C_{sy}'+ \sum_{z\in W\,:\,sz<z<y} 
M_{z,y}^s C_z'} &\quad \mbox{if $sy>y$},
\end{array}\right.\] 
where $C_s'=T_s+v^{-p_s}T_1$ and $M_{z,y}^s=\overline{M}_{z,y}^s \in A$ 
is determined as in \cite[\S 3]{Lusztig83}. 

As in \cite[\S 8]{Lusztig03}, we write $x \leftarrow_{L} y$ if there
exists some $s\in S$ such that $h_{s,y,x}\neq 0$, that is, $C_x'$ occurs
in $C_s'\, C_y'$ (when expressed in the $C'$-basis). The Kazhdan--Lusztig
left pre-order $\leq_{L}$ is the transitive closure of $\leftarrow_{L}$.
The equivalence relation associated with $\leq_{L}$ will be denoted by 
$\sim_{L}$ and the corresponding equivalence classes are called the 
{\em left cells} of $W$.

Similarly, we can define a pre-order $\leq_{R}$ by considering
multiplication by $C_s'$ on the right in the defining relation. The
equivalence relation associated with $\leq_{R}$ will be denoted by
$\sim_{R}$ and the corresponding equivalence classes are called the
{\em right cells} of $W$.  We have
\[ x \leq_{R} y \quad \Leftrightarrow \quad x^{-1} \leq_{L} y^{-1};\]
see \cite[5.6, 8.1]{Lusztig03}. Finally, we define a pre-order $\leq_{LR}$ 
by the condition that $x\leq_{LR} y$ if there exists a sequence $x=x_0,x_1,
\ldots, x_k=y$ such that, for each $i \in \{1,\ldots,k\}$, we have 
$x_{i-1} \leq_{L} x_i$ or $x_{i-1}\leq_{\cR} x_i$. The equivalence
relation associated with $\leq_{LR}$ will be denoted by $\sim_{LR}$ 
and the corresponding equivalence classes are called the {\em two-sided 
cells} of $W$.

\begin{defn} \label{defclosed} A (non-empty) subset $\Gamma$ of $W$ is 
called ''closed with respect to $\leq_L$'' if, for any $x,y\in\Gamma$,
we have $\{z\in W\mid x\leq_L z\leq_L y\}\subseteq \Gamma$. Note that 
any such subset is a union of left cells. A left cell itself is clearly 
closed with respect to $\leq_L$.
\end{defn}

Given a subset $\Gamma\subseteq W$ which is closed with respect to
$\leq_L$, we obtain an $\bH$-module $[\Gamma]_A:={\cI}_{\Gamma}/
\hat{\cI}_{\Gamma}$, where
\begin{align*}
{\cI}_{\Gamma} &:=\langle C_w'\mid w\leq_{L} 
z\mbox{ for some $z \in\Gamma$}\rangle_A,\\
\hat{\cI}_{\Gamma} &:=\langle C_w'\mid w \not\in \Gamma, w\leq_{L} z
\mbox{ for some $z \in\Gamma$}\rangle_A.
\end{align*}
Note that, by the definition of the pre-order relation $\leq_{L}$ (and
the condition that $\Gamma$ is closed with respect to $\leq_L$), these
are left ideals in $\bH$. Now denote by $e_x$ ($x\in \Gamma$) the residue
class of $C_x'$ in $[\Gamma]_A$. Then the elements $\{e_x\mid x\in \Gamma\}$
form an $A$-basis of $[\Gamma]_A$ and the action of $C_w'$ ($w \in W$) is
given by the formula
\[ C_w'.e_x=\sum_{y \in \Gamma} h_{w,x,y}\, e_y.\]
A key tool in this work will be the process of ''induction of cells''.  
Let $I\subseteq S$ and consider the parabolic subgroup $W_I\subseteq W$ 
generated by $I$. Then
\[ X_I:=\{w\in W\mid ws>w \mbox{ for all $s\in I$}\}\]
is the set of distinguished left coset representatives of $W_I$ in $W$.
The map $X_I \times W_I \rightarrow W$, $(x,u) \mapsto xu$, is a bijection
and we have $\ell(xu)=\ell(x)+\ell(u)$ for all $x\in X_I$ and $u\in W_I$; see
\cite[\S 2.1]{gepf}. Thus, given $w\in W$, we can write uniquely $w=xu$
where $x\in X_I$ and $u\in W_I$. In this case, we denote $\prI_I(w):=u$.
Let $\sim_{L,I}$ be the equivalence relation on $W_I$ for which the
equivalence classes are the left cells of $W_I$.

\begin{thm}[\protect{\cite{myind}}] \label{cellind} Let $I\subseteq S$.
If $w,w'\in W$ are such that $w\sim_L w'$, then $\prI_I(w)\sim_{L,I}
\prI_I(w')$. In particular, if $\Gamma$ is a left cell of $W_I$, then
$X_I\Gamma$ is a union of left cells of $W$.
\end{thm}

\begin{exmp} \label{cellind1} Let $\Gamma'$ be a left cell of $W_I$. Then
the subset $\Gamma:=X_I\Gamma'$ of $W$ is closed with respect to $\leq_L$.
(This immediately follows from Theorem~\ref{cellind}.) Let $\bH_I
\subseteq \bH$ be the parabolic subalgebra spanned by all $T_w$ where
$w\in W_I$. Then we obtain the $\bH_I$-module $[\Gamma']_A$, with 
standard basis $\{e_w\mid w\in\Gamma'\}$, and the $\bH$-module 
$[X_I\Gamma']_A$, with standard basis $\{e_{xw}\mid x\in X_I,w\in\Gamma'\}$.
By \cite[3.6]{myrel}, we have an isomorphism of $\bH$-modules
\[ [X_I\Gamma']_A\stackrel{\sim}{\rightarrow} \Ind_I^S([\Gamma']_A), 
\qquad e_{yv} \mapsto \sum_{x \in X_I,w\in \Gamma'} p_{xu,yv}^*\,
\big(T_x \otimes e_u\big),\]
where $p_{xu,yv}^*\in A$ are the {\em relative} Kazhdan--Lusztig
polynomials of \cite[Prop.~3.3]{myind} and, for any $\bH_I$-module $V$, we 
denote by $\Ind_I^S(V):=\bH \otimes_{\bH_I} V$ the {\em induced module}, 
with basis $\{T_x\otimes e_w \mid x\in X_I,w\in \Gamma'\}$ (see, for 
example, \cite[\S 9.1]{gepf}).
\end{exmp}

A first invariant of left cells is given as follows. For any $w\in W$, we 
denote by $\cR(w):=\{s\in S\mid ws<w\}$ the {\em right descent set} 
of~$w$. 

\begin{prop}[See \protect{\cite[2.4]{KL} for the equal parameter case and 
\cite[8.6]{Lusztig03} for the general case}] \label{klright} Let $x,y\in W$. 
If $x\sim_{L} y$, then $\cR(x)=\cR(y)$. Thus, for any $I\subseteq S$, 
the set $\{w\in W\mid \cR(w)=I\}$ is a union of left cells of $W$.
\end{prop}

We show how this can be deduced from Theorem~\ref{cellind}. Let
$x,y\in W$ be such that $x\sim_L y$. Let $s\in \cR(x)$ and set $I=\{s\}$. 
Then $\prI_I(x)=s$ and so $s=\prI_I(x)\sim_{L,I} \prI_I(y)\in W_I=\{1,s\}$. 
Since $p_s>0$, the definitions immediately show that $\{1\}$, $\{s\}$ 
are the left cells of $W_I$. Hence, we must have $\prI_I(y)=s$ and so 
$s\in \cR(y)$. Thus, we have $\cR(x)\subseteq \cR(y)$. By symmetry, we 
also have $\cR(y)\subseteq \cR(x)$ and so $\cR(x)=\cR(y)$, as required. 

\begin{defn} \label{rem2} For any $w\in W$, the {\em enhanced right descent 
set} is defined as 
\[\cR^\pi(w):=\cR(w)\cup \{sts \mid s,t\in S, st\neq ts, p_s<p_t 
\mbox{ and } wsts<w\}\]
This provides, at least, a complete invariant for the left cells of 
dihedral groups, as the following example shows.
\end{defn}

\begin{exmp} \label{celli2m}
Let $S=\{s_1,s_2\}$ and assume that $st$ has finite order $m\geq 3$. For 
$k\geq 0$ let $1_k=s_1s_2s_1\ldots$ ($k$ factors) and $2_k= s_2s_1s_2\ldots$ 
($k$ factors). Then the left cells of $W=\langle s_1, s_2 \rangle$ are 
described as follows; see Lusztig \cite[8.7, 8.8]{Lusztig03}: 
\begin{itemize}
\item[(a)] If $m$ is odd and $p_{s_1}=p_{s_2}>0$, then the left cells are 
\[\{1_0\},\quad \{2_1,1_2,2_3,\ldots,1_{m-1}\},\quad 
\{1_1,2_2,1_3,\ldots,2_{m-1}\},\quad \{2_m\}.\]
\item[(b)] If $m$ is even and $p_{s_1}=p_{s_2}>0$, then the left cells are 
\[\{1_0\},\quad \{2_1,1_2,2_3,\ldots,2_{m-1}\},\quad 
\{1_1,2_2,1_3,\ldots,1_{m-1}\},\quad \{2_m\}.\]
\item[(c)] If $m$ is even and $p_{s_2}>p_{s_1}>0$, then the left cells are 
\[\{1_0\},\quad \{2_1,1_2,2_3,\ldots,1_{m-2}\},\quad \{2_{m-1}\},\quad
\{1_1\},\quad \{2_2,1_3,2_4,\ldots,1_{m-1}\},\quad \{2_m\}.\]
\end{itemize}
By inspection of the three cases, we see that two elements $x,y\in W$ lie 
in the same left cell if and only if $\cR^\pi(x)=\cR^\pi(y)$.
\end{exmp}

\begin{cor} \label{mycor} Let $x,y\in W$. If $x\sim_L y$, then
$\cR^\pi(x)=\cR^\pi(y)$.
\end{cor}
 
\begin{proof} Assume that $x\sim_L y$. By Proposition~\ref{klright}, we 
have $\cR(x)=\cR(y)$. Let $s,t\in S$ be such that $st\neq ts$ and 
$p_s<p_t$. Let $I=\{s,t\}$ and consider the parabolic subgroup $W_I=
\langle s,t\rangle$. By Theorem~\ref{cellind}, we have $\prI_I(x) 
\sim_{L,I} \prI_I(y)$. As observed in Example~\ref{celli2m}, we have  
$xsts<x$ if and only if $ysts<y$. Consequently, we obtain
$\cR^\pi(x)=\cR^\pi(y)$. 
\end{proof}

\section{The equal parameter case} \label{seceq}

We keep the general setting of the previous section. We shall also assume 
that $\bH$ is {\em bounded} in the sense of \cite[13.2]{Lusztig03}. This is 
obviously true for all finite Coxeter groups. It also holds, for example, 
for affine Weyl groups; see the remarks following \cite[13.4]{Lusztig03}. 

\begin{defn}[Vogan \protect{\cite[3.10, 3.12]{voga}}] \label{defvog} 
For any $s,t\in S$ such that $st\neq ts$, we set 
\[ \cD_R(s,t):=\{w\in W\mid \cR(w)\cap \{s,t\} \mbox{ has exactly one
element}\}\]
and, for any $w\in \cD_R(s,t)$, we set $\cT_{s,t}(w):=\{ws,wt\}\cap 
\cD_R(s,t)$. Note that $\cT_{s,t}(w)$ consists of one or two elements; in
order to have a uniform notation, we consider $\cT_{s,t}(w)$ as a multiset 
with two identical elements if $\{ws,wt\}\cap \cD_R(s,t)$ consists of only 
one element. 

Now let $n\geq 0$ and $y,w\in W$. We define a relation $y \approx_n w$ 
inductively as follows. First, let $n=0$. Then $y\approx_0 w$ if $\cR(y)=
\cR(w)$. Now let $n>0$ and assume that $\approx_{n-1}$ has been already 
defined. Then $y\approx_n w$ if $y\approx_{n-1} w$ and if, for any $s,t
\in S$ such that $y,w\in\cD_R(s,t)$ (where $st$ has order $3$ or $4$), 
the following holds. If $\cT_{s,t}(y)=\{y_1,y_2\}$ and $\cT_{s,t}(w)
=\{w_1,w_2\}$, then either $y_1\approx_{n-1} w_1$, $y_2\approx_{n-1}w_2$ or 
$y_1\approx_{n-1} w_2$, $y_2\approx_{n-1}w_1$.

If $y\approx_n w$ for all $n\geq 0$, then $y,w$ are said to have the 
same {\em generalized $\tau$-invariant}. 
\end{defn}

\begin{rem} \label{remstrings} Let $s,t\in S$ be such that $st$ has finite 
order $m\geq 3$. Let $I=\{s,t\}$. Then the parabolic subgroup $W_I$ is 
a dihedral group of order $2m$. For any $w\in W$, the coset $wW_I$ can be 
partitioned into four subsets: one consists of the unique element $x$ of 
minimal length, one consists of the unique element of maximal length, one 
consists of the $(m-1)$ elements $xs,xst,xsts,\ldots$ and one consists of 
the $(m-1)$ elements $xt,xts, xtst,\ldots$. Following Lusztig 
\cite[10.2]{Lu1}, the last two subsets (ordered as above) are called 
{\em strings}. (Note that Lusztig considers the coset $W_Iw$ but, by 
taking inverses, the two versions are clearly equivalent.) Thus, if $w\in 
\cD_R(s,t)$, then $w$ belongs to a unique string which we denote by
$\lambda_w$. Then we certainly have
\[\cT_{s,t}(w)\subseteq\lambda_w\subseteq \cD_R(s,t) \qquad\mbox{for 
all $w\in \cD_R(s,t)$}.\]
As in \cite[10.6]{Lu1}, we set 
\[ \Gamma^*:=\bigl(\bigcup_{w\in \Gamma} \lambda_w\bigr)\setminus \Gamma 
\qquad \mbox{for any subset $\Gamma\subseteq \cD_R(s,t)$}.\]
Now assume that we are in the equal parameter case and that $\Gamma$
is a left cell of $W$ such that $\Gamma\subseteq \cD_R(s,t)$. Then
the following two results are known to hold.
\begin{itemize}
\item[(a)] If $m=3$, then $\Gamma^*$ also is a left cell; see 
Kazhdan--Lusztig \cite[Cor.~4.3]{KL}. (In this case, we have 
$\Gamma^*=\{w^* \mid w\in \Gamma\}$ where $w^*$ is the unique element
of $\cT_{s,t}(w)$.)
\item[(b)] If $m>3$, then $\Gamma^*$ is a union of at most $(m-2)$ 
left cells of $W$; see \cite[Prop.~10.7]{Lu1}. 
\end{itemize}
(For the proof of (b), it is assumed in [{\em loc.\ cit.}] that $W$ is 
crystallographic in order to guarantee certain positivity properties, but 
this assumption is now superfluous thanks to Elias--Williamson \cite{EW}.) 
\end{rem}

With these preparations, we can now state the following result which was 
originally formulated and proved by Vogan in the language of primitive 
ideals in enveloping algebras. 

\begin{prop}[Kazhdan--Lusztig \protect{\cite[\S 4]{KL}}, Lusztig 
\protect{\cite[\S 10]{Lu1}}, Vogan \protect{\cite[\S 3]{voga}}] 
\label{vogan1} Assume that we are in the equal parameter case. Let 
$\Gamma$ be a left cell of $W$. Then all elements in $\Gamma$ have 
the same generalised $\tau$-invariant.
\end{prop}

\begin{proof} If $W$ is a finite Weyl group, this follows from the results 
in \cite[\S 3]{voga}, using the known dictionary (see, e.g., Barbasch--Vogan
\cite[\S 2]{BV2}) between cells as defined in Section~\ref{sec1} and the 
corresponding notions in the theory of primitive ideals. In the general 
case, one cannot appeal to the theory of primitive ideals or other 
geometric arguments. Instead we argue as follows, using results 
from \cite[\S 4]{KL} and \cite[\S 10]{Lu1}.

We will prove by induction on $n$ that, if $y,w\in W$ are 
such that $y\sim_L w$, then $y\approx_n w$. For $n=0$, this holds by
Propositon~\ref{klright}. Now let $n>0$. By induction, we already know 
that $y\approx_{n-1} w$. Then it remains to consider $s,t \in S$ such 
that $st\neq ts$ and $y,w\in\cD_R(s,t)$.  If $st$ has order $3$, then 
Remark~\ref{remstrings}(a) shows that $\cT_{s,t}(y)=\{y^*,y^*\}$ and
$\cT_{s,t}(w)=\{w^*,w^*\}$; furthermore, $y^*\sim_L w^*$ and so
$y^*\approx_{n-1} w^*$, by induction. Now assume that $st$ has order~$4$. 
In this case, the argument is more complicated (as it is also in the 
setting of \cite[\S 3]{voga}.) Let $I=\{s,t\}$ and $\Gamma$ be 
the left cell containing $y,w$. Since all elements in $\Gamma$ have the
same right descent set, we can choose the notation such that $xs<x$ and 
$xt>x$ for all $x\in \Gamma$. Then, for $x\in \Gamma$, we have $x=x's$, 
$x=x'ts$ or $x=x'sts$ where $x'\in X_I$. This yields that 
\begin{equation*}
\cT_{s,t}(x)=\left\{\begin{array}{cl} 
\{x'st,x'st\}  & \qquad \mbox{if $x=x's$},\\
\{x't,x'tst\}  & \qquad \mbox{if $x=x'ts$},\\
\{x'st,x'st\}  & \qquad \mbox{if $x=x'sts$}.\end{array}\right.
\tag{$\dagger$}
\end{equation*}
Now we distinguish two cases.

{\em Case 1}. Assume that there exists some $x\in \Gamma$ such that $x=x's$
or $x=x'sts$. Then $\lambda_x=(x's,x'st,x'sts)$ and so $\Gamma^*$ contains 
elements with different right descent sets. Hence, by 
Remark~\ref{remstrings}(b), $\Gamma^*$ is the union of two distinct left 
cells $\Gamma_1$ and $\Gamma_2$, where we choose the notation such that:
\begin{itemize}
\item all elements in $\Gamma_1$ have $s$ in their right descent set, but
not $t$;
\item all elements in $\Gamma_2$ have $t$ in their right descent set, but
not $s$.
\end{itemize}
Now consider $y,w \in\Gamma$; we write $\cT_{s,t}(y)=\{y_1, y_2\}\subseteq 
\Gamma^*$ and $\cT_{s,t}(w)=\{w_1,w_2\}\subseteq \Gamma^*$. By 
($\dagger$), all the elements $y_1,y_1,w_1,w_2$ belong to $\Gamma_2$. In 
particular, $y_1\sim_L w_1$, $y_2\sim_L w_2$ and so, by induction, 
$y_1\approx_{n-1} w_1$, $y_2\approx_{n-1} w_2$. 

{\em Case 2}. We are not in Case~1, that is, all elements $x\in\Gamma$ 
have the form $x=x'ts$ where $x'\in X_I$. Then $\lambda_x=(x't,x'ts,
x'tst)$ for each $x\in\Gamma$. Let us label the elements in such a string 
as $x_1,x_2,x_3$. Then $x=x_2$ and $\cT_{s,t}(x)=\{x't,x'tst\}=\{x_1,x_3\}$. 

Now consider $y,w\in\Gamma$. By definition, there is a chain of elements
which connect $y$ to $w$ via the elementary relations $\leftarrow_L$, and 
vice versa. Assume first that $y,w$ are directly connected as 
$y\leftarrow_L w$. Using the labelling $y=y_2$, $w=w_2$ and the notation 
of \cite[10.4]{Lu1}, this means that $a_{22}\neq 0$. Hence, the identities 
''$a_{11}=a_{33}$'', ''$a_{13}=a_{31}$'', ''$a_{22}=a_{11}+a_{13}$'' in
\cite[10.4.2]{Lu1} imply that 
\[(y_1\leftarrow_L w_1 \mbox{ and } y_3 \leftarrow_L w_3) \quad 
\mbox{or} \quad (y_1\leftarrow_L w_3\mbox{ and } y_3\leftarrow_L w_1).\]
(See also \cite[Prop.~4.6]{shi}.) We shall write this as $\cT_{s,t}(y) 
\leftarrow_L \cT_{s,t}(w)$. Now, in general, there is a sequence of 
elements $y=y^{(0)},y^{(1)},\ldots, y^{(k)}=w$ in $\Gamma$ such that 
$y^{(i-1)} \leftarrow_L y^{(i)}$ for $1\leq i\leq k$. At each step, we have
$\cT_{s,t} (y^{(i-1)})\leftarrow_L \cT_{s,t}(y^{(i)})$ by the previous 
argument. Combining these steps, we conclude that either $y_1\leq_L w_1$, 
$y_3\leq_L w_3$ or $y_1\leq_L w_3$, $y_3\leq_L w_1$. Now, all elements in 
a string belong to the same right cell (see \cite[10.5]{Lu1}); in 
particular, all the elements 
$y_i,w_j$ belong to the same two-sided cell. Hence, \cite[Cor.~6.3]{Lu1} 
implies that either $y_1\sim_L w_1$, $y_3\sim_L w_3$ or $y_1\sim_L w_3$,
$y_3\sim_L w_1$. (Once again, the assumption in [{\em loc.\ cit.}] that
$W$ is crystallographic is now superfluous thanks to \cite{EW}.) 
Consequently, by induction, we have either $y_1\approx_{n-1} w_1$, 
$y_3\approx_{n-1}w_3$ or $y_1\approx_{n-1} w_3$, $y_3\approx_{n-1}w_1$.
\end{proof}

One of the most striking results about this invariant has been obtained
by Garfinkle \cite[Theorem~3.5.9]{gar3}: two elements of a Weyl group of 
type $B_n$ belong to the same left cell (equal parameter case) if and 
only if the elements have the same generalised $\tau$-invariant. This
fails in general; a counter-example is given by $W$ of type $D_n$ for
$n\geq 6$ (as mentioned in the introduction of \cite{gar1}).

\begin{rem} \label{rem1} Note that, if $st$ has order $m=4$, then
the set $\cT_{s,t}(w)$ may contain two distinct elements. In order to
obtain a single-valued operator, Vogan \cite[\S 4]{voga1} (for the case 
$m=4$) and Lusztig \cite[\S 10]{Lu1} (for any $m\geq 4$) propose an 
alternative construction, as follows.

Let $s,t\in S$ be such that $st$ has finite order $m\geq 3$. As in 
\cite[10.6]{Lu1}, we define an involution 
\[ \cD_R(s,t)\rightarrow \cD_R(s,t),\qquad w\mapsto \tilde{w},\]
as follows. Let $w\in\cD_R(s,t)$. Then $w$ is contained in a unique
string $\lambda_{w}$ with respect to $s,t$; see Remark~\ref{remstrings}. 
Let $i\in\{1,\ldots, m-1\}$ be the index such that $w$ is the $i$th element 
of $\lambda_{w}$. Then $\tilde{w}$ is defined to be the $(m-i)$th element 
of $\lambda_{w}$. Now let $\Gamma\subseteq \cD_R(s,t)$ be a left cell. Then 
$\tilde{\Gamma}= \{\tilde{w} \mid w\in \Gamma\}$ also is a left cell by 
\cite[Prop.~10.7]{Lu1}. (Again, it is assumed in [{\em loc.\ cit.}] that 
$W$ is crystallographic, but this is now superfluous thanks to 
\cite{EW}.) 

Hence, setting $\tilde{\cT}_{s,t}(w):=\{\tilde{w}\}$ for any $w\in
\cD_R(s,t)$, we obtain a new ''generalised $\tilde{\tau}$-invariant'' by 
exactly the same procedure as in Definition~\ref{defvog}, using 
$\tilde{\cT}_{s,t}$ instead of $\cT_{s,t}$ and allowing any $s,t\in S$ 
such that $st$ has finite order at least~$3$. 

The above procedure is the model for the more general construction of 
invariants below. As we shall see in Example~\ref{explu1}, this even 
provides a new proof---which does not rely on \cite{EW}---for the fact 
that the map $w\mapsto\tilde{w}$ preserves left cells.
\end{rem}

\section{An abstract setting for generalised $\tau$-invariants} \label{sec2}

We keep the general setting of Section~\ref{sec1}, where $\pi=\{p_s\mid 
s\in S\}$ are positive weights for $W$. 

\begin{defn} \label{mydef1} A pair $(I,\delta)$ consisting of a subset
$I\subseteq S$ and a map $\delta \colon W_I\rightarrow W_I$ is called
{\it admissible} if the following conditions are satisfied for every
left cell $\Gamma'\subseteq W_I$ (with respect to the weights $\{p_s\mid 
s\in I\}$):
\begin{itemize}
\item[(1)] The restriction of $\delta$ to $\Gamma'$ is injective and 
$\delta(\Gamma')$ also is a left cell.
\item[(2)] The map $\delta$ induces an $\bH_I$-module isomorphism 
$[\Gamma']_A\cong [\delta(\Gamma')]_A$.
\end{itemize}
We say that $(I,\delta)$ is {\it strongly admissible} if, in addition 
to (1) and (2), the following condition is satisfied:
\begin{itemize}
\item[(3)] We have $u\sim_{R,I}\delta(u)$ for all $u\in W_I$. 
\end{itemize}
The map $\delta$ has a canonical extension to a map $\tilde{\delta}\colon 
W\rightarrow W$: Given $w\in W$, we write $w=xu$ where $x\in X_I$ and 
$u\in W_I$; then we set $\tilde{\delta}(w):= x\delta(u)$. 
\end{defn}

The situation considered by Kazhdan--Lusztig \cite[\S 4]{KL} fits into
this setting as follows.

\begin{exmp} \label{expkalu1} Let $I=\{s,t\}$ with $s\neq t$ and $st$ of 
order~$3$; then $W_I$ is isomorphic to the symmetric group $\fS_3$. The 
left cells of $W_I$ are easily determined; they are
\[ \Gamma_1':=\{1\}, \qquad \Gamma_s':=\{s,ts\}, \qquad \Gamma_t':=\{t,st\}, 
\qquad \Gamma_0':=\{sts\}.\]
The matrix representation of $\bH_I$ afforded by $[\Gamma_s']_A$ with 
respect to the basis $\{e_{s}, e_{ts}\}$ is given by (where we set
$p:=p_s=p_t>0$):
\[ C_{s}' \mapsto \left[\begin{array}{cc} v^p+v^{-p} & 1 \\ 0 & 0
\end{array}\right],\qquad C_{t}' \mapsto \left[\begin{array}{cc} 0 & 0
\\ 1 & v^p+v^{-p} \end{array}\right],\]
and we obtain exactly the same matrices when we consider the matrix
representation afforded by $[\Gamma_t']_A$ with respect to the basis
$\{e_{st},e_{t}\}$. (See \cite[7.2, 7.3, 8.7]{Lusztig03} where dihedral 
groups in general are considered.) Thus, the conditions
(1), (2), (3) in Definition~\ref{mydef1} hold for $(I,\delta)$, if we 
define $\delta \colon W_I \rightarrow W_I$ as follows:
\[\delta\colon \quad \begin{array}{cccccc} 
1 & s & t & st & ts & sts \\ \downarrow & \downarrow & \downarrow & 
\downarrow & \downarrow & \downarrow \\  1 & st & ts & s & t & sts
\end{array}\]
We notice that, if $w\in W$ is such that $w\in \cD_R(s,t)$
(see Definition~\ref{defvog}), then $\{\tilde{\delta}(w)\}=\{ws,wt\} 
\cap \cD_R(s,t)$, hence $\tilde{\delta}(w)=w^*$ with the notation of 
\cite[\S 4]{KL}.
\end{exmp}

\begin{prop} \label{myprop} Let $(I,\delta)$ be an admissible pair. Then 
the following hold.
\begin{itemize}
\item[(a)] If $\Gamma$ is a left cell of $W$, then so is 
$\tilde{\delta}(\Gamma)$ (where $\tilde{\delta}$ is the canonical 
extension of $\delta$ to $W$) and $\tilde{\delta}$ induces an $\bH$-module 
isomorphism $[\Gamma]_A\cong [\tilde{\delta} (\Gamma)]_A$.
\item[(b)] If $(I,\delta)$ is strongly admissible, then we 
have $w\sim_R \tilde{\delta}(w)$ for all $w\in W$.
\end{itemize}
\end{prop}

\begin{proof} (a) By Theorem~\ref{cellind}, there is a left cell 
$\Gamma'$ of $W_I$ such that $\Gamma\subseteq X_I\Gamma'$. By condition (1) 
in Definition~\ref{mydef1}, the set $\Gamma_1':=\delta(\Gamma')$ also is a 
left cell of $W_I$ and, by condition (2), the map $\delta$ induces an
$\bH_I$-module isomorphism $[\Gamma']_A\cong [\Gamma_1']_A$. By 
Example~\ref{cellind1}, the subsets $X_I\Gamma'$ and $X_I\Gamma_1'$
of $W$ are closed with respect to $\leq_L$ and, hence, we have 
corresponding $\bH$-modules $[X_I\Gamma']_A$ and $[X_I\Gamma_1']_A$.
These two $\bH$-modules are isomorphic to the induced modules
$\Ind_I^S([\Gamma'])$ and $\Ind_I^S([\Gamma_1'])$, respectively,
where explicit isomorphisms are given by the formula in 
Example~\ref{cellind1}. Now, by \cite[Lemma~3.8]{myrel}, we have
\[ p_{xu,yv}^*=p_{xu_1,yv_1}^* \qquad \mbox{for all $x,y\in X_I$ and
$u,v\in \Gamma'$},\]
where we set $u_1=\delta(u)$ and $v_1=\delta(v)$ for $u,v\in\Gamma'$. 
By \cite[Prop.~3.9]{myrel}, this implies that $\tilde{\delta}$ maps the 
partition of $X_I\Gamma'$ into left cells of $W$ onto the analogous 
partition of $X_I\Gamma_1'$. In particular, since $\Gamma\subseteq 
X_I\Gamma'$, the set $\tilde{\delta}(\Gamma) \subseteq X_I\Gamma_1'$ 
is a left cell of $W$; furthermore, \cite[Prop.~3.9]{myrel} also shows
that $\tilde{\delta}$ induces an $\bH$-module isomorphism
$[\Gamma]_A\cong [\tilde{\delta}(\Gamma)]_A$. 

(b) Since condition (3) in Definition~\ref{mydef1} is assumed to hold, this
is just a restatement of \cite[Prop.~9.11(b)]{Lusztig03}.
\end{proof}

As a first consequence, we can now show that \cite[Cor.~4.3]{KL} (concerning
the Kazhdan--Lusztig star operations) holds for general weight functions. 
(Partial results in this direction are obtained in 
\cite[Cor.~3.5(4)]{shi1}.) Note that some work has to be done to obtain
this generalisation since, in the setting of \cite[\S 4]{KL}, the 
polynomials $M_{y,w}^s$ are constant, which is no longer true in the general 
case and so some new arguments are required.

\begin{cor} \label{star1} Let $s,t\in S$ be such that $st$ has order $3$. 
Then, for any $w\in\cD_R(s,t)$, there is a unique $w^*\in \cD_R(s,t)$ such 
that $\cT_{s,t}(w)=\{w^*,w^*\}$ (as in \cite[\S 4]{KL} and 
Definition~\ref{defvog}). Let $\Gamma \subseteq \cD_R(s,t)$ be a left cell 
(with respect to the given weights $\{p_s\mid s\in S\}$). Then $\Gamma^*:=
\{w^*\mid w\in \Gamma\}$ also is a left cell. Furthermore, the map 
$w\mapsto w^*$ induces an $\bH$-module isomorphism $[\Gamma]_A \rightarrow 
[\Gamma^*]_A$ and we have $w\sim_{R} w^*$ for all $w\in \Gamma$.
\end{cor}

\begin{proof} Let $I=\{s,t\}$ and define $\delta\colon W_I\rightarrow
W_I$ as in Example~\ref{expkalu1}. We already noted that then 
$\tilde{\delta}(w) =w^*$ for all $w\in \cD_R(s,t)$. Hence, the assertions 
follow from Proposition~\ref{myprop}. 
\end{proof}

In analogy to Definition~\ref{defvog}, we can now introduce an 
invariant of left cells as follows.
 
\begin{defn} \label{mydef} Let $\Delta$ be a collection of admissible 
pairs $(I,\delta)$ as in Definition~\ref{mydef1}. For each $I\subseteq S$ 
which occurs as a first component of a pair in $\Delta$, we assume that 
we are given a relation $\Lambda_I\subseteq W_I\times W_I$ which contains
the relation defined by~$\sim_{L,I}$. (For example, $\Lambda_I=\{(u,u')
\in W_I\times W_I\mid \cR(u)=\cR(u')\}$; see Proposition~\ref{klright}.)

Now let $n\geq 0$ and $y,w\in W$. Then we define a relation 
$y \leftrightharpoons_n w$ inductively as follows. 
\begin{itemize}
\item[(i)] For $n=0$, we have $y\leftrightharpoons_0 w$ if $(\prI_I(y),
\prI_I(w))\in\Lambda_I$ for all $(I,\delta)\in \Delta$. 
\item[(ii)] Now let $n>0$ and assume that $\leftrightharpoons_{n-1}$ has 
been already defined. Then $y\leftrightharpoons_n w$ if 
$y\leftrightharpoons_{n-1} w$ and $\tilde{\delta}(y) 
\leftrightharpoons_{n-1} \tilde{\delta}(w)$ for all $(I,\delta)\in 
\Delta$.
\end{itemize}
If $y\leftrightharpoons_n w$ for all $n\geq 0$, then $y,w$ are said to have 
the same {\em generalized $\tilde{\tau}^\Delta$-invariant}. 
\end{defn}

\begin{cor} \label{mythm} In the setting of Definition~\ref{mydef},
all elements in a left cell $\Gamma$ of $W$ have the same generalised 
$\tilde{\tau}^\Delta$-invariant.
\end{cor}

\begin{proof} We prove by induction on $n$ that, if $y,w\in W$ are such 
that $y\sim_L w$, then $y\leftrightharpoons_n w$. For $n=0$, this holds
by Theorem~\ref{cellind}. Now assume that $n>0$. By induction, we 
already know that $y\leftrightharpoons_{n-1} w$. Then it remains to 
consider a pair $(I,\delta)\in \Delta$. By Proposition~\ref{myprop}(a), 
we have $\tilde{\delta}(y)\sim_L \tilde{\delta}(w)$ and, by induction, 
we have $\tilde{\delta}(y)\leftrightharpoons_{n-1}\tilde{\delta}(w)$.
\end{proof}

The situation considered by Lusztig \cite[\S 10]{Lu1} (see also 
Vogan \cite[\S 4]{voga1} and McGovern \cite[\S 4]{mcg} for the case 
$m=4$) fits into this setting as follows.

\begin{exmp} \label{explu1} Let $\fI_{\geq 3}$ be the set of all subsets 
$I\subseteq S$ such that $I=\{s,t\}$, where $s\neq t$, $p_s=p_t$ 
and $st$ has finite order $m\geq 3$. For any $I\in \fI_{\geq 3}$, the 
group $W_I$ is a dihedral group of order $2m$. For $k\geq 0$ let 
$1_k= sts\ldots$ ($k$ factors) and $2_k= tst\ldots$ ($k$ factors). Then
the left cells of $W_I$ are described as follows (see Example~\ref{celli2m}):
\begin{align*}
\{1_0\},\quad &\{2_1,1_2,2_3,\ldots,1_{m-1}\},\quad 
\{1_1,2_2,1_3,\ldots,2_{m-1}\},\quad \{2_m\}\qquad \mbox{($m$ odd)},\\
\{1_0\},\quad &\{2_1,1_2,2_3,\ldots,2_{m-1}\},\quad 
\{1_1,2_2,1_3,\ldots,1_{m-1}\},\quad \{2_m\}\qquad \mbox{($m$ even)}.
\end{align*}
We define an involution $\delta\colon W_I\rightarrow W_I$ as follows:
\[ \delta(1_0)=1_0, \quad \delta(2_m)=2_m,\quad \delta(1_k)=1_{m-k},
\quad \delta(2_k)=2_{m-k} \quad \mbox{for $1\leq k\leq m-1$}.\]
If $m$ is odd, then $\delta$ perserves each of the left cells $\{1_0\}$,
$\{2_m\}$ and interchanges the two left cells with $m-1$ elements (reversing
the order in which the elements are listed). If $m$ is even, then $\delta$ 
perserves each of the four left cells of $W_I$, where in each of the two 
left cells with $m-1$ elements, the order of the elements is reversed. 
Thus, conditions (1) and (3) in Definition~\ref{mydef1} hold for all
pairs in the collection
\[\Delta_{\geq 3}:=\{(I,\delta)\mid I\in\fI_{\geq 3}\}.\]
Using the formulae in \cite[7.2, 7.3]{Lusztig03}, it is straightforward to 
check that condition (2) also holds. For $m=3$, this has been done explicitly
in Example~\ref{expkalu1}. Let us also show explicitly how this works 
for $m=4, 5$. 

First let $m=4$. Consider the two left cells $\Gamma_s'=\{s,ts,
sts\}$ and $\Gamma_t'=\{t,st,tst\}$. The matrix representation afforded by 
$[\Gamma_s']_A$ with respect to the basis $\{e_{s}, e_{ts}, e_{sts}\}$ is 
given by:
\[ C_{s}' \mapsto \left[\begin{array}{ccc} v^p+v^{-p} & 1 & 0 \\ 0 & 0 & 0\\
0 & 1 & v^p+v^{-p} \end{array}\right],\qquad C_{t}' \mapsto 
\left[\begin{array}{ccc} 0 & 0 & 0 \\ 1 & v^p+v^{-p} & 1 \\ 0 & 0 & 0
\end{array}\right] \]
(where $p:=p_s=p_t$). The matrix representation afforded by $[\Gamma_t']_A$ 
with respect to the basis $\{e_{t}, e_{st}, e_{tst}\}$ is given by:
\[ C_{s}' \mapsto \left[\begin{array}{ccc} 0 & 0 & 0 \\ 1 & v^p+v^{-p} & 1\\
0 & 0 & 0 \end{array}\right],\qquad C_{t}' \mapsto \left[\begin{array}{ccc} 
v^p+v^{-p} & 1 & 0 \\ 0 & 0 & 0 \\ 0 & 1 & v^p+v^{-p} \end{array}\right]. \]
Thus, there is no bijection $\Gamma_s'\rightarrow \Gamma_t'$ which 
induces an $\bH_I$-module isomorphism $[\Gamma_s']_A\cong [\Gamma_t']_A$.
However, we have $\delta(\Gamma_s')=\Gamma_s'$ where $s\mapsto sts$, 
$ts\mapsto ts$, $sts\mapsto s$, and this map yields a non-trivial 
$\bH_I$-module automorphism of $[\Gamma_s']_A$; a similar
remark applies to $[\Gamma_t']$.

Now let $m=5$. We have the two left cells $\Gamma_s'=\{s,ts,sts,tsts\}$ 
and $\Gamma_t'=\{t,st,tst,stst\}$. The matrix representation afforded by 
$[\Gamma_s']_A$ with respect to the basis $\{e_{s}, e_{ts}, e_{sts},
e_{tsts}\}$ is  given by:
\[ C_{s}' \mapsto \left[\begin{array}{cccc} v^p+v^{-p} & 1 & 0 & 0\\ 
0 & 0 & 0 & 0 \\ 0 & 1 & v^p+v^{-p} & 1 \\ 0 & 0 & 0 & 0 \end{array}\right],
\qquad C_{t}' \mapsto \left[\begin{array}{cccc} 0 & 0 & 0 &0 \\ 
1 & v^p+v^{-p} & 1 &0 \\ 0 & 0 & 0 & 0 \\ 0 & 0 & 1 & v^p+v^{-p} 
\end{array}\right],\]
and we obtain exactly the same matrices when we consider the matrix
representation afforded by $[\Gamma_t']_A$ with respect to the basis
$\{e_{stst},e_{tst},e_{st},e_t\}$. 

We notice that, if $w\in W$ is any element such that $w\in \cD_R(s,t)$
(see Definition~\ref{rem2}), then $\tilde{\delta}(w)=\tilde{w}$,
with $\tilde{w}$ as defined in Remark~\ref{rem1}. Thus, 
Proposition~\ref{myprop} provides a new proof of the part of 
\cite[Prop.~10.7]{KL} concerning the tilde construction; this new
proof does not rely on the positivity properties used in [{\em loc.\ cit.}].
\end{exmp}

Finally, we consider a genuine case of unequal parameters.

\begin{exmp} \label{myexp} Let $\fI_{\pi}$ be the set of all subsets 
$I\subseteq S$ such that $I=\{s,t\}$, where $s\neq t$, $p_s<p_t$ 
and $st$ has finite even order $m\geq 4$. For any $I\in \fI_{\pi}$, 
the group $W_I$ is a dihedral group of order $2m$. For $k\geq 0$ let 
$1_k= sts\ldots$ ($k$ factors) and $2_k= tst\ldots$ ($k$ factors). Then 
the left cells of $W_I$ are described as follows (see Example~\ref{celli2m}):
\[\{1_0\},\quad \{2_1,1_2,2_3,\ldots,1_{m-2}\},\quad \{2_{m-1}\},\quad
\{1_1\},\quad \{2_2,1_3,2_4,\ldots,1_{m-1}\},\quad \{2_m\}.\]
We define an involution $\delta\colon W_I \rightarrow W_I$ as follows:
$\delta(w) =w$ for $w\in \{1_0,1_1,2_{m-1},2_m\}$ and 
\[\delta\colon \quad\begin{array}{ccccc} 2_1 & 1_2 & 2_3 & \ldots &1_{m-2}\\
\updownarrow & \updownarrow & \updownarrow & \ldots & \updownarrow \\
2_2 & 1_3 & 2_4 & \ldots & 1_{m-1}\end{array}\]
Thus, $\delta$ perserves each of the left cells $\{1_0\}$, $\{2_{m-1}\}$,
$\{1_1\}$, $\{2_m\}$ and interchanges the two left cells with $m-2$ elements
(preserving the order in which the elements are listed). So, conditions (1) 
and (3) in Definition~\ref{mydef1} hold for all pairs in the collection
\[\Delta_\pi:=\{(I,\delta) \mid I\in\fI_\pi\}.\]
Using the knowledge of the polynomials $M_{y,w}^s$ (see 
\cite[Exc.~11.4]{gepf} or \cite[7.5, 7.6]{Lusztig03}), it is 
straightforward to check that condition (2) also holds. Let us show 
explicitly how this works for $m=4,6$. 

First let $m=4$. We have to consider the two left cells $\Gamma_1'=
\{t,st\}$ and $\Gamma_2'=\{ts,sts\}$. The matrix representation afforded 
by $[\Gamma_1']_A$ with respect to the basis $\{e_{t}, e_{st}\}$ is given by:
\[ C_{s}' \mapsto \left[\begin{array}{cc} 0 & 0 \\ 1 & v^{p_s}+v^{-p_s}
\end{array}\right],\qquad C_{t}' \mapsto \left[\begin{array}{cc} v^{p_t}+
v^{-p_t} & v^{p_t-p_s}+v^{p_s-p_t} \\ 0 & 0 \end{array}\right],\]
and we obtain exactly the same matrices when we consider the matrix
representation afforded by $[\Gamma_2']_A$ with respect to the basis
$\{e_{ts},e_{sts}\}$. 

Next consider the case $m=6$. We have the two left cells $\Gamma_1'=
\{t,st,tst,stst\}$ and $\Gamma_2'=\{ts,sts,tsts,ststs\}$. The two matrices 
describing the action of $C_s'$ and $C_t'$ on $[\Gamma_1']_A$ with respect 
to the basis $\{e_t,e_{st},e_{tst},e_{stst}\}$ are given by
\[\left[\begin{array}{c@{\hspace{4pt}}c@{\hspace{4pt}}
c@{\hspace{4pt}}c} 
0 & 0 & 0 & 0  \\
1 & v^{p_s}{+}v^{-p_s} & 0 & 0 \\ 
0 & 0 & 0 & 0 \\
0 & 0 & 1 & v^{p_s}{+}v^{-p_s} \end{array}\right],\qquad 
\left[\begin{array}{c@{\hspace{5pt}}c@{\hspace{5pt}}c@{\hspace{5pt}}c} 
v^{p_t}{+}v^{-p_t} & v^{p_t-p_s}{+}v^{p_s-p_t} & 0 & 1 \\ 
0 & 0 & 0 & 0 \\ 
0 & 1 & v^{p_t}{+}v^{-p_t} & v^{p_t-p_s}{+}v^{p_s-p_t}\\ 
0 & 0 & 0 & 0 \end{array}\right],\]
respectively, and we obtain exactly the same matrices when we consider the 
matrix representation afforded by $[\Gamma_2']_A$ with respect to the basis
$\{e_{ts},e_{sts},e_{tsts},e_{ststs}\}$. 
\end{exmp}

For any subset $I=\{s,t\}\subseteq S$ where $s\neq t$ and $st$ has finite 
order $m\geq 3$, we set $\Lambda_I=\{(u,u')\in W_I\times W_I\mid \cR^\pi
(u)=\cR^\pi(u')\}$; see Definition~\ref{rem2}. With this convention, 
we would now like to state the following conjecture. 

\begin{conj} \label{myconj} Two elements $y,w\in W$ belong to the same
left cell (with respect to the given weights) if and only if $y,w$ belong
to the same two-sided cell and $y,w$ have the same generalised 
$\tilde{\tau}^{\Delta}$-invariant where $\Delta=\Delta_{\geq 3}
\cup \Delta_{\pi}$ (see Examples~\ref{explu1} and \ref{myexp}).
\end{conj}

If $W$ is finite and we are in the equal parameter case, then 
Conjecture~\ref{myconj} is known to hold except possibly in type $B_n,D_n$;
see the remarks at the end of \cite[\S 6]{geha}. We have checked that
the conjecture also holds for $F_4$, $B_n$ ($n\leq 7$) and all 
possible weights, using {\sf PyCox} \cite{pycox}. 

By considering collections $\Delta$ with subsets $I\subseteq S$ of 
size bigger than~$2$, one can obtain further refinements of the above 
invariants. In particular, it is likely that the results of
Bonnaf\'e and Iancu \cite{bo2}, \cite{BI} can be interpreted in terms
of generalised $\tau^\Delta$-invariants for suitable collections $\Delta$.
This will be discussed elsewhere.


\end{document}